\documentclass[12pt]{amsart}

\usepackage{amstext,amsfonts,amssymb,amscd,amsbsy,amsmath,verbatim}
\usepackage{extarrows}
\usepackage{color,tikz,multirow}
\usepackage{amsthm}
\usepackage{latexsym}
\usepackage[all]{xy}
\usepackage{enumerate}
\usepackage[bookmarks, colorlinks=true, linkcolor=blue, citecolor=blue, urlcolor=blue, backref=page]{hyperref}
\usepackage[lite,nobysame]{amsrefs} 
\usepackage[margin=1in]{geometry}

\newtheorem{lemma}{Lemma}[section]

\newtheorem{prop}[lemma]{Proposition}
\newtheorem{cor}[lemma]{Corollary}
\newtheorem{conj}[lemma]{Conjecture}

\newtheorem{claim*}{Claim}

\newtheorem{problem}[lemma]{Problem}
\newtheorem{notation}[lemma]{Notation}
\newtheorem{question}[lemma]{Question}

\theoremstyle{definition}

\theoremstyle{remark}

\newtheorem{example}[lemma]{Example}
\newtheorem{remark}[lemma]{Remark}

\newcommand{\Tor}{\operatorname{Tor}}

\newcommand{\kk}{\Bbbk}
\newcommand{\coker}{\operatorname{coker}}

\newcommand{\codim}{\operatorname{codim}}
\newcommand{\PP}{\mathbb{P}}

\newcommand{\ZZ}{\mathbb{Z}}
\newcommand{\QQ}{\mathbb{Q}}

\newcommand{\cO}{\mathcal{O}}
\newcommand{\cE}{\mathcal{E}}
\newcommand{\cF}{\mathcal{F}}

\newcommand{\fH}{\mathfrak{H}}
\newcommand{\fg}{\mathfrak{g}}
\newcommand{\gl}{\mathfrak{gl}}
\newcommand{\rU}{\mathrm{U}}
\newcommand{\rH}{\mathrm{H}}

\newcommand{\FF}{\mathbf{F}}

\newcommand{\Sym}{\operatorname{Sym}} 
\newcommand{\GL}{{GL}}

\newcommand{\defi}[1]{\textsf{#1}} 
\newcommand{\arXiv}[1]{{\tt \href{http://arxiv.org/abs/#1}{arXiv:#1}}}

\newcommand{\BQ}{\mathrm{B}}
\newcommand{\DD}{\mathrm{D}}

\newcommand{\CvbQ}{\mathrm{C}_{\text{vb}}}

\def\BS{Boij--S\"oderberg~}

\title{Questions about Boij--S\"oderberg theory}

\date{June 6, 2016}

\author{Daniel Erman}
\address{Department of Mathematics, University of Wisconsin, Madison, WI 53706, USA}
\email{derman@math.wisc.edu}
\urladdr{\url{http://www.math.wisc.edu/~derman/}}

\author{Steven V Sam}
\address{Department of Mathematics, University of Wisconsin, Madison, WI 53706, USA}
\email{svs@math.wisc.edu}
\urladdr{\url{http://www.math.wisc.edu/~svs/}}

\thanks{DE was partially supported by NSF DMS-1302057 and  SS was partially supported by NSF DMS-1500069.}


\begin{document}

\maketitle

\section{Background on Boij--S\"oderberg Theory}~\label{subsec:background BS}
\defi{Boij--S\"oderberg theory} focuses on the properties and duality relationship between two types of numerical invariants.  One side involves the Betti table of a graded free resolution over the polynomial ring.  The other side involves the cohomology table of a coherent sheaf on projective space.  The theory began with a conjectural description of the cone of Betti tables of finite length modules, given in~\cite{boij-sod1}.  Those conjectures were proven in~\cite{eis-schrey1}, which also described the cone of cohomology tables of vector bundles and illustrated a sort of duality between Betti tables and cohomology tables.

The theory itself has since expanded in many directions: allowing modules whose support has higher dimension, replacing vector bundles by coherent sheaves, working over rings other than the polynomial ring, and so on.  But at its core, Boij--S\"oderberg theory involves:
\begin{enumerate}
	\item  A classification, up to scalar multiple, of the possible Betti tables of some class of objects (for example, free resolutions of finitely generated modules of dimension $\leq c$).
	\item  A classification, up to scalar multiple, of the cohomology tables of some class of objects (for examples, coherent sheaves of dimension $\leq n-c$).
	\item  Intersection theory-style duality results between Betti tables and cohomology tables.
\end{enumerate}

One motivation behind Boij and S\"oderberg's conjectures was the observation that it would yield an immediate proof of the Cohen--Macaulay version of the Multiplicity Conjectures of Herzog--Huneke--Srinivasan~\cite{herzog-srinivasan}.  Eisenbud and Schreyer's~\cite{eis-schrey1} thus yielded an immediate proof of that conjecture, and the subsequent papers~\cite{boij-sod2,eis-schrey2} provided a proof of the Multiplicity Conjecture for non-Cohen--Macaulay modules.  Other applications of the theory involve Horrocks' Conjecture~\cite{erman-beh}, cohomology of tensor products of vector bundles~\cite{eis-schrey-banks}, sparse determinantal ideals~\cite{boocher}, concavity of Betti tables \cite{mccullough-concavity} and more.  In fact, \BS theory has grown into an active area of research in commutative algebra and algebraic geometry.  See \S\ref{sec:more topics} or \cite{floystad-expository, survey2} for a summary of many of the related papers.  In addition, many features from the theory have been implemented via Macaulay2 packages such as \verb BoijSoederberg.m2,  \verb BGG.m2,  and \verb TensorComplexes.m2 ~\cite{M2}. 

In this paper, we focus on discussing several open questions related to Boij--S\"oderberg theory. Of course, the choice of topics reflects our own bias and perspective. We also briefly review some of the major aspects of the theory, but those interested in a fuller expository treatment should refer to~\cite{eis-schrey-icm} or ~\cite{floystad-expository}.

\subsection{Betti tables}
Let $S=\kk[x_0, \dots, x_n]$ with the grading $\deg(x_i)=1$ for all $i$, and with $\kk$ any field.  Let $M$ be a finitely generated graded module over $S$.  Since $M$ is graded, it admits a minimal free resolution $\FF=[F_0\gets F_1 \gets \dots \gets F_p\gets 0]$.  The \defi{Betti table} of $M$ is a vector whose coordinates $\beta_{i,j}(M)$ encode the numerical data of the minimal free resolution of $M$.  Namely, since each $F_i$ is a graded free module, we can let $\beta_{i,j}(M)$ be the number of degree $j$ generators of $F_i$; equivalently, we can write $F_i = \bigoplus_{j\in \ZZ} S(-j)^{\beta_{i,j}}$; also equivalently, the Betti numbers come from graded Tor groups with respect to the residue field: $\beta_{i,j}(M):=\dim \Tor_i(\FF,\kk)_j$.  

For example, if $S=\kk[x_0,x_1]$ and $M=S/(x_0,x_1^2)$ then the minimal free resolution of $M$ is a Koszul complex
\[
F = S \longleftarrow \begin{matrix}S^1(-1)\\ \oplus\\ S^1(-2) \end{matrix} \longleftarrow S^1(-3)\longleftarrow 0.
\]
The Betti table of $M$ is traditionally displayed as the following array or matrix:
\[
\beta(M) = \begin{bmatrix}
\beta_{0,0}&\beta_{1,1}&\dots&\beta_{p,p}\\
\beta_{0,1}&\beta_{1,2}&\dots &\beta_{p,p+1}\\
\vdots & \ddots &&\vdots 
\end{bmatrix}.
\]
In the example above, we thus have
\[
\beta(M) = \begin{bmatrix}
1&1&-\\
-&1&1
\end{bmatrix}.
\]

There is a huge literature on the properties of Betti tables, and we refer the reader to \cite{eis-syzygy} as a starting point. Yet there are also many fundamental open questions about Betti tables. The most notable question is Horrocks' Conjecture, due to Horrocks~\cite[Problem 24]{hartshorne-vector} and Buchsbaum-Eisenbud \cite[p. 453]{buchs-eis-gor}, which proposes that the Koszul complex is the ``smallest'' free resolution.  More precisely, one version of the conjecture proposes that if $c=\codim M$, then $\sum_{i,j} \beta_{i,j}(M)\geq 2^c$.  The conjecture is known in five variables and other special cases~\cite{avramov-buchweitz-betti,chara-evans-problems} but remains wide open in general.

Boij and S\"oderberg proposed taking the convex cone spanned by the Betti tables, and then focusing on this cone.  This is like studying Betti tables ``up to scalar multiple'', which would remove the subtleties behind questions like Horrocks' Conjecture. Note that convex combinations are natural in this context, as $\beta(M\oplus M')=\beta(M)+\beta(M')$.  We define $B^c(S)$ as the convex cone spanned by the Betti tables of all $S$-modules of codimension $\geq c$:
\[
B^c(S) := \QQ_{\geq 0}\{ \beta(M) \mid \codim M \geq c\} \subseteq \bigoplus_{i=0}^{n+1} \bigoplus_{j\in \ZZ}\QQ.
\]

Boij and S\"oderberg's original conjectures described the cone $B^{n+1}(S)$, which is the case of finite length $S$-modules~\cite{boij-sod1}.  This description was based on the notion of a \defi{pure resolution of type $d=(d_0,\dots,d_{n+1})\in \ZZ^{n+2}$}, which is an acyclic complex where the $i$'th term is generated entirely in degree $d_i$; in other words, it is a minimal free complex of the form:
\[
S(-d_0)^{\beta_{0,d_0}}\longleftarrow S(-d_1)^{\beta_{1,d_1}}\longleftarrow S(-d_2)^{\beta_{2,d_2}}\longleftarrow \dots \longleftarrow S(-d_p)^{\beta_{n+1,d_{n+1}}}\longleftarrow 0.
\]
Any such resolution must satisfy $d_0<d_1<\dots<d_{n+1}$, and Boij and S\"oderberg conjectured that for any strictly increasing vector $(d_0,\dots,d_{n+1})$ there was such a pure resolution.  It was known that if such a resolution existed, then the vector $d$ determined a unique ray in $B^{n+1}(S)$~\cite[\S 2.1]{boij-sod1}, and Boij and S\"oderberg conjectured that these were precisely the extremal rays of $B^{n+1}(S)$.  So the extremal rays of $B^{n+1}(S)$ should be in bijection with strictly increasing vectors (sometimes called \defi{degree sequences}) $d=(d_0,\dots,d_{n+1})$.  With the advent of hindsight, this is a natural guess: pure resolutions will give the Betti tables with the fewest possible nonzero entries, and so they will always produce extremal rays.

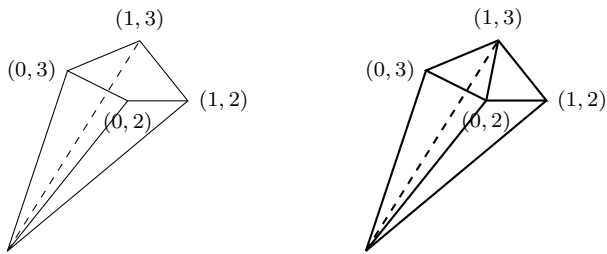
\begin{figure}
\begin{tikzpicture}[scale=0.8]
\draw[-](0,0)--(1,3);
\draw[left] (1,3) node {{\tiny $(0,3)$}};
\draw[-](0,0)--(2,2.5);
\draw[below] (2,2.5) node {{\tiny $(0,2)$}};
\draw[-](0,0)--(3,2.5);
\draw[right] (3,2.5) node {{\tiny $(1,2)$}};
\draw[dashed,-](0,0)--(2.2,3.5);
\draw[above] (2.2,3.5) node {{\tiny $(1,3)$}};
\draw[-](1,3)--(2,2.5)--(3,2.5)--(2.2,3.5)--cycle;
\end{tikzpicture}
\hspace{1cm}
\begin{tikzpicture}[scale=0.8]
\draw[-,thick](0,0)--(1,3);
\draw[left] (1,3) node {{\tiny $(0,3)$}};
\draw[-,thick](0,0)--(2,2.5);
\draw[below] (2,2.5) node {{\tiny $(0,2)$}};
\draw[-,thick](0,0)--(3,2.5);
\draw[right] (3,2.5) node {{\tiny $(1,2)$}};
\draw[dashed,-,thick](0,0)--(2.2,3.5);
\draw[-,thick](2,2.5)--(2.2,3.5);
\draw[above] (2.2,3.5) node {{\tiny $(1,3)$}};
\draw[-,thick](1,3)--(2,2.5)--(3,2.5)--(2.2,3.5)--cycle;
\end{tikzpicture}
\caption{\footnotesize For a cone of Betti tables, the extremal rays come from Cohen--Macaulay modules with pure resolutions.  The rays may thus be labelled by strictly increasing sequences of integers, as illustrated on the left.  The cone has a simplicial fan structure, where simplices correspond to increasing chains of integers with respect to the partial order, as illustrated on the right.}
\label{fig:cones}
\end{figure}

There is a natural termwise partial order on such vectors, where $d\leq d'$ if $d_i\leq d_i'$ for all $i$, and Boij and S\"oderberg also conjectured that this partial order endowed $B^{n+1}(S)$ with the structure of a simplicial fan, where $B^{n+1}(S)$ is the union of simplicial cones and the simplices correspond to maximal chains of degree sequences with respect to the partial order.  For instance, the cone in Figure~\ref{fig:cones} decomposes as the union of two simplicial cones; the two cones correspond to the chains $(0,2)<(1,2)<(1,3)$ and $(0,2)<(0,3)<(1,3)$.  

The existence of pure resolutions  was first proven in~\cite{efw} in characteristic zero and in~\cite{eis-schrey1} in arbitrary characteristic; further generalizations appear in~\cite{beks-tensor,floystad-triplets}.  The rest of Boij and S\"oderberg's conjectures were proven in~\cite{eis-schrey1} and the theory has since been extended from finite length modules to finitely generated modules ~\cites{boij-sod2} and to bounded complexes of modules~\cite{eis-erman-cat}.

One of the most striking corollaries of Boij--S\"oderberg theory is the resulting decomposition of Betti tables.  The simplicial structure of the cone of Betti tables provides an algorithm for writing a Betti table $\beta(M)$ as a positive, rational sum of the Betti tables of pure resolutions.  For instance, returning to our example of $M=\kk[x,y]/(x,y^2)$, we have:
\begin{equation}\label{eqn:xy2}
\beta(M) = \begin{bmatrix}
1&1&-\\
-&1&1
\end{bmatrix} =
\frac{1}{3} \begin{bmatrix}
2&3&-\\
-&-&1
\end{bmatrix}
+
\frac{1}{3} \begin{bmatrix}
1&-&-\\
-&3&2
\end{bmatrix}
\end{equation}
The decomposition of $\beta(M)$ will always be a finite sum, which is an immediate consequence of the fact that every Betti table has only finitely many nonzero entries, and there are thus a finite number of extremal rays which could potentially contribute to that Betti table.  This decomposition algorithm is implemented in Macaulay2, and it is extremely efficient.

These decompositions provide a counterintuitive aspect of Boij--S\"oderberg theory. 

\begin{example}\label{ex:xy2}
Let $S=\kk[x,y]$ and $M=S/(x,y^2)$.
The pure resolutions appearing on the right-hand side of \eqref{eqn:xy2} are resolutions of actual modules.  If we let $M'=S/(x,y)^2$ and $M'':=\coker \begin{pmatrix}x&y&0\\0&x&y\end{pmatrix}$ then we have:
\[
\beta(M) = \tfrac{1}{3}\beta(M')+\tfrac{1}{3}\beta(M'').
\]
Yet the above equation is purely numerical: Boij and S\"oderberg did not address whether the free resolution of $M$ could be built from the free resolutions of $M'$ and $M''$, and Eisenbud and Schreyer's proof provides no such categorification. We discuss this in more detail in \S\ref{sec:categorification}.
\end{example}

If we study finitely generated modules which are not necessarily of finite length, then the description of the cone of Betti tables is a bit more subtle.  For the cone $B^c(S)$ of Betti tables of modules of modules of codimension $\geq c$, the extremal rays still correspond to pure resolutions of Cohen--Macaulay modules, but now the modules may have codimension between $c$ and $n+1$.  For instance, if we consider the non-Cohen--Macaulay, cyclic module $M=\kk[x,y,z]/(x^2, xy, xz, yz)$, then the Boij--S\"oderberg decomposition will be:
\[
\beta(M) = \begin{bmatrix}
1&-&-&-\\
-&4&4&1
\end{bmatrix} =
\frac{1}{3} \begin{bmatrix}
1&-&-&-\\
-&6&8&3
\end{bmatrix}
+
\frac{2}{3} \begin{bmatrix}
1&-&-\\
-&3&2
\end{bmatrix}.
\]
The Betti tables on the right can be realized as the Betti tables of $S/(x_1,x_2,x_3)^2$ and $S/(x_1,x_2)^2$, which are Cohen--Macaulay modules of codimension $3$ and $2$, respectively.  We denote these as pure resolutions of type $(0,2,3,4)$ and type $(0,2,3,\infty)$ respectively, because under this convention, the termwise partial order still induces a simplicial fan structure on $B^c(S)$.

Even more generally, we could leave the world of modules and look instead at bounded complexes of free modules or equivalently at elements of $\DD^b(S)$.  The corresponding cone is yet again spanned by pure resolutions of Cohen--Macaulay modules, but where we also allow homological shifts.  With the right conventions, the whole story carries over in great generality in this context, including the decomposition theorem.  For instance, on $S=\kk[x,y]$ we could take the complex
\[
\FF := \left[S^1\overset{\left(\begin{smallmatrix}x&y\end{smallmatrix}\right)}{\xlongleftarrow{\hspace*{1.1cm}}} S^2(-1)\overset{\left(\begin{smallmatrix}-y^2&xy\\xy&-x^2\end{smallmatrix}\right)}{\xlongleftarrow{\hspace*{1.1cm}}} S^2(-3)\overset{\left(\begin{smallmatrix}y\\x\end{smallmatrix}\right)}{\xlongleftarrow{\hspace*{1.1cm}}} S^1(-4)\right],
\]
which has finite length homology $H_{0}\FF = \kk,\ H_{1}\FF = \kk(-2)$.  We then have the decomposition:
\[
\beta(\FF)=\begin{bmatrix} 1&2&-&-\\-&-&2&1\end{bmatrix}=
\frac{1}{2}\begin{bmatrix}
-&1&-&-\\
-&-&3&2
\end{bmatrix}
+\frac{1}{2}
\begin{bmatrix}
2&3&-&-\\
-&-&1&-
\end{bmatrix}.
\]
See~\cite{eis-erman-cat} for more on Boij--S\"oderberg theory for complexes.

\subsection{Cohomology tables}
On the sheaf cohomology side, we fix a coherent sheaf $\cE$ on $\PP^n_{\kk}$.  We define the \defi{cohomology table} of 
$\cE$ as a vector whose coordinates $\gamma_{i,j}(\cE)$ encode graded sheaf cohomology groups of $\cE$.  In particular, $\gamma_{i,j}(\cE) := \dim_{\kk} \rH^{i}(\PP^n,\cE(j))$.

When displaying cohomology tables, we follow the tradition introduced in~\cite{eis-floy-schrey} and write:
\[
\gamma(\cE) = \begin{matrix}
\dots & \gamma_{n,-n-2}&\gamma_{n,-n-1}&\gamma_{n,-n}&\gamma_{n,-n+1}&\dots&\\
\dots & \gamma_{n-1,-n-1}&\gamma_{n-1,-n}&\gamma_{n-1,-n+1}&\gamma_{n-1,-n+2}&\dots&\\
&  &\ddots && \\
\dots & \gamma_{1,-3}&\gamma_{1,-2}&\gamma_{1,-1}&\gamma_{1,0}&\dots&\\
\dots & \gamma_{0,-2}&\gamma_{0,-1}&\gamma_{0,0}&\gamma_{0,1}&\dots&\\
\end{matrix}
\]
Thus for instance, if $\cE=\cO_{\PP^2}$ then we have:
\[
\gamma(\cO_{\PP^2}) = 
\begin{matrix}
\dots & 3&1&-&-&-&-&\dots&\\
\dots & -&-&-&-&-&-&\dots&\\
\dots &-&-&1&3&6&10&\dots&\\
\end{matrix}
\]
where the $3$ on the bottom row corresponds to $\gamma_{0,1}(\cO_{\PP^2}(1))=\dim \rH^0(\PP^2,\cO_{\PP^2}(1))=3$.  

While sheaf cohomology is central to all of modern algebraic geometry, the research on cohomology tables is nowhere near as extensive as the work on Betti tables.  In fact, cohomology tables seem to have first appeared in~\cite{eis-floy-schrey}, where the cohomology table of a sheaf $\cF$ is written as the Betti table of the corresponding Tate resolution over the exterior algebra.  

As with Betti tables, there are many open questions related to cohomology tables.  One notable question is whether there exists a non-split rank $2$ vector bundle on $\PP^n$ for any $n\geq 5$.  This question could be answered entirely with information in the cohomology table.   A coherent sheaf on $\PP^n$ is a vector bundle if and only if there are only finitely many zero entries in the cohomology table in the rows corresponding to $\rH^i$ where $i=1,2,\dots,n-1$; and it is a sum of line bundles if and only if all of those intermediate entries are zero.  And since the cohomology table is a refinement of the Hilbert polynomial, the rank of the vector bundle can be read off from the cohomology table as well.

Inspired by the Boij--S\"oderberg conjectures, Eisenbud and Schreyer introduced the convex cone spanned by the cohomology tables.  We start by focusing on vector bundles; define $\CvbQ(\PP^n)$ as the convex cone spanned by the cohomology tables of all vector bundles on $\PP^n$:
\[
\CvbQ(\PP^n) := \QQ_{\geq 0}\{ \gamma(\cE) \mid \cE \text{ is a vector bundle on $\PP^n$}\} \subseteq \prod_{i=0}^{n+1} \prod_{j\in \ZZ}\QQ.
\]

Eisenbud and Schreyer give a complete description of this cone in~\cite{eis-schrey1}.  A \defi{supernatural bundle}  is a vector bundle with as few nonzero cohomology groups as possible; more precisely, $\cE$ is supernatural of type $f=(f_1,\dots,f_{n})\in \ZZ^{n}$ if
\[
\rH^i(\PP^n,\cE(j))\ne 0 \iff \begin{cases}
i=n \text{ and } j<f_n\\
1\leq i \leq n-1 \text{ and } f_{i+1}<j<f_{i}\\
i=0 \text{ and } j>f_1.
\end{cases}
\]
Any such bundle must satisfy $f_1>f_2>\dots>f_n$ and such a strictly decreasing sequence of integers is called a \defi{root sequence}.  Eisenbud and Schreyer proved that there exists a supernatural bundle with any specified root sequence; as was the case with pure resolutions, the vector $f$ thus determines a unique ray in $\CvbQ(\PP^n)$, and Eisenbud and Schreyer proved that these were precisely the extremal rays of that cone.  Again with the advent of hindsight, this is a natural guess: supernatural bundles give the cohomology tables with the fewest possible nonzero entries, and so they will always produce extremal rays.

The termwise partial order on root sequences then induces a simplicial fan structure on $\CvbQ(\PP^n)$ in a manner entirely parallel to the story for the cone of Betti tables.
\begin{example}\label{ex:PP2 example}
Let $\cE$ be the cokernel of a generic map $\phi \colon \cO_{\PP^2}(-2)^2\to \cO_{\PP^2}(-1)^5$.  Then $\cE$ is a rank $3$ bundle on $\PP^2$.  A computation in Macaulay2 yields the decomposition
\begin{align*}
\gamma(\cE) &= \begin{bmatrix}
\dots & 20&10&3&-&-&-&-&\dots\\
\dots & -&-&1&2&-&-&-&\dots\\
\dots &-&-&-&-&5&13&24&\dots\\
\end{bmatrix}\\
&=
\frac{1}{3}\begin{bmatrix}
\dots & 15&6&-&-&-&-&-&\dots\\
\dots & -&-&3&3&-&-&-&\dots\\
\dots &-&-&-&-&6&15&27&\dots\\
\end{bmatrix}
+
\begin{bmatrix}
\dots & 15&8&3&-&-&-&-&\dots\\
\dots & -&-&-&1&-&-&-&\dots\\
\dots &-&-&-&-&3&8&15&\dots\\
\end{bmatrix},
\\
\intertext{where the tables used in the decomposition can be realized as cohomology tables of supernatural bundles.  In fact, if $\cE':=(\Sym_2\Omega^1_{\PP^2})(2)$ and $\cE'':=\Omega^1_{\PP^2}(1)$, then we have}
&= \tfrac{1}{3}\gamma(\cE')+\tfrac{1}{3} \gamma(\cE'').
\end{align*}
We return to this example in \S\ref{sec:categorification}.
\end{example}

All of the results about vector bundles can be extended to coherent sheaves as well.  Let $C(\PP^n)$ be the cone of cohomology tables of coherent sheaves $\cF$  on $\PP^n$.  In~\cite{eis-schrey2}, Eisenbud and Schreyer describe this cone in detail. Every extremal ray comes from a sheaf $\cE$, which is supported on a linear subspace $\PP^a\subseteq \PP^n$ for some $a\leq n$, and where $\cE$ is a supernatural vector bundle on $\PP^a$.  As in the case of vector bundles, the extremal rays thus correspond to root sequences, and an appropriate partial order induces a simplicial fan structure on $C(\PP^n)$. But there is one key difference for coherent sheaves: the decomposition might involve an infinite number of summands~\cite[Example~0.3]{eis-schrey2}.  There are open questions even in this case.
\begin{question}
Under what conditions will an infinite sum of supernatural sheaves correspond to a cohomology table, at least up to scalar multiple?
\end{question}

\section{Categorification}\label{sec:categorification}
The most natural and mysterious question raised by \BS theory has to do with categorifying the decompositions into pure/supernatural tables.
\begin{question}
\begin{enumerate}[\rm \indent (1)]
\item  For a graded module $M$, does the decomposition of $\beta(M)$ into pure Betti tables ``lift'' to a corresponding decomposition of  the module $M$ itself?  
\item For a vector bundle $\cE$, does the decomposition of $\gamma(\cE)$ into supernatural cohomology tables ``lift'' to a corresponding decomposition of the bundle $\cE$ itself?
\end{enumerate}
\end{question}
Let us return to Examples~\ref{ex:xy2} and \ref{ex:PP2 example}, and continue with that notation.  On the module side, we had the decomposition
\[
\beta(M) = \tfrac{1}{3} \beta(M')+ \tfrac{1}{3}\beta(M'')
\]
and on the vector bundle side we had
\begin{align*}
\gamma(\cE) &= \tfrac{1}{3}\gamma \left(\cE'\right)+\gamma(\cE'').
\end{align*}
If we want to decompose $M$ or $\cE$ into pieces corresponding to the Boij--S\"oderberg decomposition, then there are three challenges:
\begin{enumerate}
	\item {\bf Denominators:}  Is there a natural way to systematically clear the denominators that arise in these sorts of decompositions? 
	\item  {\bf Moduli:} Modules with pure resolutions and supernatural bundles can have nontrivial moduli.  Namely, there can be a lot of modules/vector bundles with the same pure Betti table/supernatural cohomology table, or even the same table up to scalar multiple.\footnote{For example, if $\phi$ is any sufficiently general $2\times 4$ matrix of linear forms on $S=\kk[x,y,z]$, then the cokernel of $\phi$ will define a pure resolution with degree sequence $(0,1,3,4)$.  Varying $\phi$ produces non-isomorphic pure resolutions.  Similar constructions work for supernatural bundles~\cite[\S6]{eis-schrey-banks}.}   How do we select the appropriate module or bundle to represent the summands appearing in the decomposition?
	\item  {\bf Filtration:}  Even if we overcome the first two obstacles, we still need to assemble these pieces, perhaps through a filtration.  How do we build such a filtration?
\end{enumerate}

On the module side, almost nothing is known.  Eisenbud, Erman and Schreyer obtain categorifications under highly restrictive hypotheses in~\cite{ees-filtering}, but their techniques avoid the issues raised by denominators and moduli.  

The example above is already challenging.  We might clear denominators by replacing $M$ by $M^{\oplus 3N}$ for some $N$.  This would yield:
\[
\beta(M^{\oplus 3N}) = \begin{pmatrix}2N&3N&-\\-&-&N\end{pmatrix} + \begin{pmatrix}N&-&-\\-&3N&2N\end{pmatrix}.
\]
The diagrams on the righthand correspond to unique modules.  Continuing with the notation of Example~\ref{ex:xy2}, the diagrams on the right must correspond to $(M')^{\oplus N}$ and $(M'')^{\oplus N}$.  We would thus want to realize $M^{\oplus 3N}$ as an extension of $(M')^{\oplus N}$ and $(M'')^{\oplus N}$.  Since there are no nonzero extensions of $M'$ by $M''$, the only possibility is an extension
\[
0\to (M')^{\oplus N} \to M^{\oplus N} \to (M'')^{\oplus N}\to 0.
\]
As originally observed by Sam and Weyman~\cite{sam-weyman}, this is also impossible: $M^{\oplus N}$ is annihilated by the linear form $x$ whereas the submodule $(M')^{\oplus N}$ is not annihilated by any linear form.

However, on the vector bundle side (in characteristic $0$), some recent progress was made.  In fact, the first two issues now seem to have a satisfactory answer.  The most naive idea for clearing  denominators is to take direct sums of $\cE$, which is equivalent to applying $\cE\mapsto \cE\otimes \cO^{\oplus N}$.  Instead, we introduce a collection of Fourier--Mukai transforms which can be thought of as ``twisting'' by something like an Ulrich sheaf.  These transforms clear denominators in exactly the way one would hope for~\cite[Corollary~1.4]{erman-sam-diagonals}.  In addition, these functors also resolve the moduli issue, as they send every supernatural bundle to a direct sum of the $\GL$-equivariant supernatural bundles~\cite[Corollary~1.10]{erman-sam-diagonals}.

In the case of Example~\ref{ex:PP2 example}, this is sufficient to categorify the decomposition.  In fact, after applying an appropriate Fourier--Mukai transform $\Phi$, we obtain
\[
\Phi(\cE) \cong \cE' \oplus \cE''^{\oplus 3}.
\]

However, beyond the cases covered in~\cite[Corollary~1.9]{erman-sam-diagonals}, there is no general procedure for producing a filtration from the various supernatural pieces.

\begin{conj}
Fix a vector bundle $\cE$ on $\PP^n$.  Let $f$ be the root sequence corresponding to the first step of the Boij--S\"oderberg decomposition of $\cE$ and let $\cF_f$ be the equivariant supernatural bundle of type $f$ constructed in~\cite[Theorem~6.2]{eis-schrey2}.  

Then there exists $r>0$, a Fourier--Mukai transform $\Phi$, and a short exact sequence
\[
0\to \cF_f^{\oplus r} \to \Phi(\cE) \to \cE'\to 0
\]
where:
\begin{enumerate}[\indent \rm (1)]
\item  The cohomology table of $ \Phi(\cE)$ is a scalar multiple of the cohomology table of $\cE$, and
\item  $\gamma(\cE)=\gamma(\cF_f^{\oplus r})+\gamma(\cE')$.
\end{enumerate}
\end{conj}
In essence, the conjecture says that the Fourier--Mukai transforms constructed in~\cite{erman-sam-diagonals} can be used to categorify the first step of the Boij--S\"oderberg decomposition of any vector bundle.  If true, one could iterate this to categorify the entire decomposition.

We can ask for similar results for modules.  Even finding a natural categorification in the single case of Example~\ref{ex:xy2} would represent significant progress.

\begin{problem}
Find an operation, other than direct sums, which naturally clears denominators on the module side.
\end{problem}

\section{\BS theory and the tails of infinite resolutions}\label{subsec:beyond}
The first major progress in extending \BS theory to other projective varieties was due to Eisenbud and Schreyer, who observed that the cone of cohomology tables depends only mildly on the projective scheme~\cite[Theorem~5]{eis-schrey-abel}. In particular, they show that if $X\subseteq \PP^n$ is a projective, $d$-dimensional subscheme which has an Ulrich sheaf, then the cone of cohomology tables on $X$ (with respect to $\cO_X(1))$ equals the cone of cohomology tables for $\PP^d$.  This provides further motivation for Eisenbud and Schreyer's question about the existence of Ulrich sheaves~\cite[p. 543]{eis-schrey-chow}.

\begin{question}[Eisenbud-Schreyer]
Does every $X\subseteq\PP^n$ have an Ulrich sheaf?
\end{question}

Eisenbud and Erman extended \BS theory to many other graded rings, but only if one restricts to Betti tables of perfect complexes with finite length homology~\cite[Theorem~0.8]{eis-erman-cat}.  Their result is similar to the Eisenbud and Schreyer result, but with a linear Koszul complex replacing the Ulrich sheaf.  Namely, they show that if $R$ is a graded ring of dimension $d$ which has $d$ independent linear forms, then the cone of Betti tables of perfect complexes $\FF$ where all homology modules of $\FF$ have finite length lines up with the corresponding cone of Betti tables for $\kk[x_1,\dots,x_d]$.

Restricting to perfect complexes makes the situation much simpler, as one entirely avoids the complications involved in the study of infinite resolutions~\cite{avramov1,peeva-book}.  While the cone of Betti tables of resolutions has been worked out for some other graded rings~\cites{bbeg, kummini-sam, gheorghita-sam}, these are all fairly simple rings which avoid many of the complexities of infinite resolutions since they have finite Cohen--Macaulay representation type (i.e., there are only finitely many isomorphism classes of indecomposable maximal Cohen--Macaulay modules). Some natural next cases here are other rings of finite Cohen--Macaulay representation type: quadric hypersurfaces and the Veronese surface.

In addition, recent work of Avramov, Gibbons, and Wiegand provides the first examples beyond the case of finite Cohen--Macaulay representation type. More than just the cone, they work out the full semigroup of Betti tables for graded, Gorenstein algebras with Hilbert series $1+es+s^2$ for some $e$; when $e\geq 3$, these rings have wild representation type.

\begin{remark}
In \cite{kummini-sam, gheorghita-sam}, the extremal rays of the \BS cone are still pure free resolutions. These can be constructed for the quadric hypersurface $\kk[x_0,\dots,x_n]/(x_0^2)$ when $\kk$ is a field of characteristic $0$ by noting that this ring is the symmetric algebra on a super vector space with $1$ odd variable $x_0$ and $n$ even variables $x_1,\dots,x_n$ and then applying \cite[Theorem 0.1]{efw}. However, we have been unable to deform these resolutions to get pure free resolutions over higher rank quadrics.  
\end{remark}

An alternate approach, and one which is common in the study of infinite resolutions, would be to focus on the tails of the infinite resolutions.  In other words, we could define a cone of Betti tables where we only focus on modules $M$ which arise as ``sufficiently high'' syzygy modules.  The notion of ``sufficiently high'' will depend on the context, though~\cite[\S7]{eisenbud-peeva} provides a notion in the case where $R$ is a complete intersection ring, and this is a natural starting point.

\begin{problem}
Let $R$ be a graded complete intersection $R=S/(f_1,\dots,f_c)$ and let $\BQ^{\text{tail}}(R)$ be the cone of Betti tables $\beta(M)$ where $M$ is a sufficiently high syzygy module.  Describe $\BQ^{\text{tail}}(R)$.
\end{problem}

Again, there are many natural variants of this problem, such as where $R$ is: a Golod ring, a ring defined by a toric or monomial ideal, or the homogeneous coordinate ring of a curve under a very positive embedding.  In a different direction, see~\cite[Conjecture~1.6]{beks-local} for an open question about the local ring case.

\section{Exact sequences}
The long exact sequence in cohomology provides a central tool for studying the cohomology of a coherent sheaf, and it could thus be extremely fruitful to develop machinery that explores Boij--S\"oderberg theory in exact sequences.  The main challenge is that cohomology tables (or Betti tables) are subadditive, but not necessarily additive, over short exact sequences.

Suppose we have a short exact sequence of coherent sheaves on $\PP^n$
\[
0\to \mathcal A\to \cE \to \mathcal B\to 0
\]
where we fully understand the sheaf cohomology of $\mathcal A$ and $\mathcal B$.  A good example to keep in mind is the case where $\mathcal A$ and $\mathcal B$ are supernatural bundles.  After twisting by $\cO(j)$ we get a long exact sequence in cohomology:
\[
\cdots \to \rH^i(\cE(j)) \to \rH^i(\mathcal B(j)) \to \rH^{i+1}(\mathcal A(j))\to \rH^{i+1}(\mathcal \cE(j))\to \cdots
\]
If all connecting maps in cohomology are zero, then this maximizes all entries in the cohomology table of $\cE$ in the sense that $\rH^i(\cE(j)) \cong \rH^i(\mathcal A(j)) \oplus \rH^i(\mathcal B(j))$. If this holds for all $j$ then we have $\gamma(\cE)=\gamma(\mathcal A)+\gamma (\mathcal B)$. On the other hand, if some of the connecting maps were nonzero, then this reduces the cohomology table of $\gamma(\cE)$ via ``consecutive cancellations'', i.e., by simultaneously reducing $\rH^i(\cE(j))$ and $\rH^{i+1}(\cE(j))$ by the same amount.\footnote{The terminology of consecutive cancellations is common in the literature on Betti tables, but does not appear to have been used in discussion of sheaf cohomology.}

We thus expect some ambiguities in the cohomology table of an extension bundle like $\cE$, and here is where Boij--S\"oderberg theory might help.  Some patterns of consecutive cancellations will yield tables which lie inside the cone $C(\PP^n)$ of genuine cohomology tables, whereas other patterns might lie outside of that cone.  The following notation will be helpful.
\begin{notation}
For a root sequence $f$ we write $\sigma_f$ for the supernatural cohomology table of rank one and with root sequence $f$.  If $f$ has one entry, so $f=(i)$, then we simply use $\sigma_i$.
\end{notation}

\begin{example}\label{ex:triangle}
Consider a rank $10$ extension bundle:
\[
0\to \cO^{5}_{\PP^1}(-2)\to \cE \to \cO^{5}_{\PP^1}(2)\to 0.
\]
The maximal possible cohomology table is if the extension splits:
\[
\begin{matrix}
\cdots & 50&40&30&20&15&10&5&-&-&-& \cdots\\
\cdots & -&-&-&5&10&15&20&30&40&50 & \cdots
\end{matrix}.
\]
However, there might be some consecutive cancellations, and taking into account the symmetry imposed by Serre duality, the possibilities are:
\[
\begin{matrix}
\cdots & 50&40&30&20&15-a&10-b&5-a&-&-&- & \cdots\\
\cdots & -&-&-&5-a&10-b&15-a&20&30&40&50 & \cdots
\end{matrix}.
\]
But which pairs $a$ and $b$ can actually arise? It takes a bit of work to answer this. If we want the cohomology table to remain inside the cone then we get:
\[
10-b\geq 2(5-a); \qquad \text{equivalently, } 2a\geq b.
\]
This turns out to be the only constraint, and the range of possibilities for the cohomology table of $\cE$ can thus be described by a triangle in the cone of cohomology tables, with corners corresponding to $(a,b)=(0,0), (5,5)$ and $(5,10)$.  The corners of that triangle correspond to sums of supernatural tables, as illustrated in Figure~\ref{fig:triangle}.
\end{example}
\begin{figure}
\begin{tikzpicture}[scale=0.75]
\draw[->](0,0)--(0,5);
\draw[->](0,0)--(5,0);
\draw (2.5,-.5) node {{$a$}};
\draw (-.5,2.5) node {{$b$}};
\draw[-,thick](0,0)--(2.5,2.5)--(2.5,5)--cycle;
\fill[lightgray](0,0)--(2.5,2.5)--(2.5,5)--cycle;
\draw (0,0) node {$\bullet$};
\draw[right] (2.5,2.5) node {{$5\sigma_{-3}+5\sigma_1$}};
\draw (2.5,2.5) node {$\bullet$};
\draw[below] (0,0) node {{$5\sigma_{-2}+5\sigma_0$}};
\draw (-.5,2.5) node {{$b$}};
\draw[right] (2.5,5) node {{$10\sigma_{-1}$}};
\draw (2.5,5) node {$\bullet$};
\end{tikzpicture}\caption{\footnotesize The potential cohomology tables for an extension bundle form a polytope in the cone of cohomology tables; with $\mathcal E$ as in Example~\ref{ex:triangle}, the polytope is a triangle.}
\label{fig:triangle}
\end{figure}
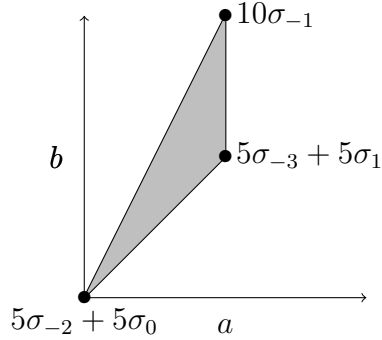

More generally, if $\cE$ is an extension of two sheaves $\mathcal A$ and $\mathcal B$, then the possible cohomology tables of $\cE$ is a polytope inside of $C(\PP^n)$.

\begin{problem}
Fix a short exact sequence of vector bundles
\[
0\to \mathcal A\to \cE \to \mathcal B\to 0.
\]
\begin{enumerate}[\indent \rm (1)]
\item  Assume that $\mathcal A$ is supernatural of type $f$ and $\mathcal B$ is supernatural of type $f'$.  Describe the polytope of possible cohomology tables of $\cE$ in terms of supernatural tables $\sigma_g$ for various root sequences $g$.
\item  For more general bundles $\mathcal A$ and $\mathcal B$, describe the polytope of possible cohomology tables of $\cE$ in terms of the Boij--S\"oderberg decompositions of $\mathcal A$ and $\mathcal B$.
\end{enumerate}
\end{problem}

There are many natural variants of this question.  For instance, one could aim to describe the potential cohomology tables of $\mathcal A$ in terms of $\mathcal E$ and $\mathcal B$; one could work with longer exact sequences; or one could pose analogous questions about Betti tables.

\section{Boij--S\"oderberg theory over a DVR}

Fix a coherent sheaf $\cF$ on $\PP^n_{\ZZ_p}$ that is flat over $\ZZ_p$.  How much can the cohomology table of $\cF$ ``jump'' when you pass from the generic point $\QQ_p$ to the special fiber $\mathbb F_p$?  Are there any restrictions other than having a constant Hilbert polynomial?  A version of Boij--S\"oderberg theory with $\ZZ_p$-coefficients would offer insights into this question.

More generally, fix a DVR $R$ and let $S:=R[x_0,\dots,x_n]$, where $\deg(x_i)=1$ for all $i$. In place of finite length modules, consider graded $S$-modules which are finite rank, free $R$-modules.  

\begin{problem}
Let $S=R[x_0,\dots,x_n]$ with $R$ a DVR.  Describe:
\begin{enumerate}[\indent \rm (1)]
\item  The cone of Betti tables of finitely generated, graded $S$-modules $M$ which are flat and finitely generated as $R$-modules.
\item  The cone of cohomology tables of vector bundles on $\PP^n_{R}$.
\end{enumerate}
\end{problem}

The functorial approach in \cite{eis-erman-cat} offers a possible approach.

\begin{problem} With notation as above:
\begin{enumerate}[\indent \rm (1)]
\item Extend the construction of the functor $\Phi$ from~\cite[\S2]{eis-erman-cat} to a pairing of derived categories:
\[
\Phi_R\colon \DD^b(R[x_0, \dots, x_n]) \times \DD^b(\PP^n_R)\to \DD^b(R[t]).
\]
\item Give explicit formulas, extending \cite[Theorem~2.3]{eis-erman-cat}, that relate the numerical invariants of $\Phi_R(\FF,\cE)$ to the numerical invariants of $\FF$ and $\cE$.
\end{enumerate}
\end{problem}

Even more generally, one might allow $R$ to be a regular local ring of dimension $>1$.  However when $R$ is a DVR, the target of the functor $\Phi$ has global dimension $2$.  Since much more is known about Betti tables in this context, we expect that the case where $R$ is a DVR is more tractable.

\section{Non-commutative analogues}

Consider the degree sequence $(0,1,3,4)$. The Herzog--K\"uhl equations state that any finite length module over a polynomial ring in $3$ variables with a pure resolution of type $(0,1,3,4)$ is a multiple of the following table:
\begin{align} \label{eqn:0134}
\begin{bmatrix} 1 & 2 & & \\ & & 2 & 1 \end{bmatrix}.
\end{align}
One can easily deduce that this Betti table is non-realizable: this would be the Betti table of $\kk[x,y,z]/I$ where $I$ is generated by two linear forms, which must then have a linear Koszul relation.

However, there is a way to realize this as the Betti table of a finite length module if we are willing to replace $\kk[x,y,z]$ by another algebra. In particular, define a $3$-dimensional Lie algebra $\fH$ (Heisenberg Lie algebra) with basis $\{x,y,z\}$ and the following multiplication
\[
[x,y] = z, \qquad [x,z] = [y,z] = 0.
\]
Recall that given any Lie algebra $\fg$, its universal enveloping algebra $\rU(\fg)$ is the tensor algebra on $\fg$ modulo the relations $x \otimes y - y \otimes x = [x,y]$ for all $x,y \in \fg$. In general, $\rU(\fg)$ is only a filtered algebra, and the associated graded algebra is the symmetric algebra $\Sym(\fg)$ by the Poincar\'e--Birkhoff--Witt theorem. However, if $\fg$ is graded, then the same is true for $\rU(\fg)$. In our case, $\fH$ is graded via $\deg(x) = \deg(y) = 1$ and $\deg(z) = 2$. In this case, minimal free resolutions of graded modules are well-defined. 

The Chevalley--Eilenberg complex (which becomes the usual Koszul complex under the PBW degeneration mentioned above) always gives a free resolution of the residue field, but is not necessarily minimal. For $\fH$, the Chevalley--Eilenberg complex looks like:
\[
0 \to \rU(\fH)(-4) \to \begin{array}{c} \rU(\fH)(-2) \\ \oplus\\ \rU(\fH)(-3)^{\oplus 2}
\end{array} \to 
\begin{array}{c} \rU(\fH)(-1)^{\oplus 2} \\ \oplus\\ \rU(\fH)(-2) \end{array} \to \rU(\fH).
\]
The two terms of degree $2$ cancel since it corresponds to the redundancy $xy - yx = z$. So we get the following minimal free resolution:
\[
0 \to \rU(\fH)(-4) \to \rU(\fH)(-3)^{\oplus 2} \to \rU(\fH)(-1)^{\oplus 2} \to \rU(\fH)
\]
and hence we can realize \eqref{eqn:0134}.

This suggests that non-realizable integral points in the Boij--S\"oderberg cone may be realizable over $\rU(\fg)$ for a $\ZZ_{>0}$-graded Lie algebra $\fg$. Note that a finite-dimensional $\ZZ_{>0}$-graded Lie algebra $\fg$ is necessarily nilpotent since $[\fg, \fg_i] \subset \fg_{i+1}$. Also, a standard graded polynomial ring in $n$ variables is $\rU(\kk^n)$ where $\kk^n$ is the abelian Lie algebra of dimension $n$ concentrated in degree $1$.

In dimension $2$, every nilpotent Lie algebra is abelian, and in dimension 3, the only possibilities are $\kk^3$ and $\fH$ (to normalize, we insist that the generators of the Lie algebra consist of degree 1 elements). It is easy to see that every integral point on a pure ray in the Boij--S\"oderberg cone in $2$ variables is realizable, so we offer the following question:

\begin{question}
For every degree sequence $(d_0,d_1,d_2,d_3)$, and every integral point on the corresponding ray in the Boij--S\"oderberg cone, is there a finite length module either over $\kk[x,y,z]$ or $\rU(\fH)$ whose Betti table is that integral point?
\end{question}

Of course, there is a natural extension for any number of variables:

\begin{question}
For every degree sequence $(d_0,\dots,d_n)$, and every integral point on the corresponding ray in the Boij--S\"oderberg cone, does there exist an $n$-dimensional $\ZZ_{>0}$-graded Lie algebra $\fg$ generated in degree 1, and a finite length module over $\rU(\fg)$ whose Betti table is that integral point?
\end{question}

One source of examples for $\ZZ_{>0}$-graded Lie algebras are the nilpotent radicals of parabolic subalgebras of (split) reductive Lie algebras. For example, parabolic subalgebras of $\gl_n$ are subalgebras of block upper-triangular matrices (where the block sizes are fixed) the nilpotent radical is the subalgebra where the diagonal blocks are identically $0$. Over a field of characteristic $0$, Kostant's version of the Borel--Weil--Bott theorem (see \cite[\S 2]{sam} for a convenient reference and combinatorial details for the symplectic Lie algebra which will be mentioned below) calculates the Tor groups of the restriction of an irreducible representation from the reductive Lie algebra to the nilpotent one. In fact, the EFW complexes are a special case of Kostant's calculation where the reductive Lie algebra is $\gl_{n+1}$ and we take the nilpotent radical of the subalgebra of block upper-triangular matrices with block sizes $1$ and $n$ (in this case, the nilpotent radical is abelian). 

Furthermore, one can construct many kinds of pure resolutions using this construction, though not necessarily all degree sequences are realizable (in the EFW case, they are). We point out that $\fH$ can be realized as the nilpotent radical for a parabolic subalgebra of the symplectic Lie algebra $\mathfrak{sp}_4$ (see \cite[\S 2.2]{sam}), and using this representation-theoretic perspective, one can realize more integral points which are not realizable over $\kk[x,y,z]$ than the example presented above. Details will appear in forthcoming work of the second author.

\begin{remark}
We are appealing to the fact that $\rU(\fg)$ ``looks like'' a graded polynomial algebra, which is made precise by the Poincar\'e--Birkhoff--Witt theorem. One could also study other non-commutative algebras that look like graded polynomial algebras. One desirable property is finite global dimension. As a starting point, one may consider Artin--Schelter algebras, see \cite{artin-schelter}.
\end{remark}

\section{Connection with Stillman's Conjecture}
Let $\widehat{S}:=\kk[x_0,x_1,x_2,\dots]$ and fix positive integers $e_1,\dots,e_r$.  Let $I$ be an ideal with generators in degrees $e_1,\dots,e_r$.  Stillman's Conjecture asks whether there is an upper bound for the projective dimension of $\widehat{S}/I$ which depends only on $e_1,\dots,e_r$; in particular, the bound should not depend on the ideal $I$~\cite[Problem~3.14]{peeva-stillman}.  

Caviglia has shown that a positive answer to Stillman's Conjecture is equivalent to the existence of an upper bound on the regularity of $\widehat{S}/I$ which depends only on $e_1,\dots,e_r$~\cite[Theorem 29.5]{peeva-book}.  This allows us to rephrase Stillman's conjecture as a finiteness statement about Betti tables.  For simplicity, we focus on the case where all of the $e_i$ have the same value $e$. Then Stillman's conjecture is equivalent to the following:

\begin{conj}[Alternate Stillman]
There are only finitely many Betti tables $\beta(M)$ of $\widehat{S}$-modules $M$ satisfying:
\[
\beta_{0,j}(M)=\begin{cases}
1 & j=0\\
0 & j\ne 0
\end{cases}
\qquad
\text{ and }
\qquad
\beta_{1,j}(M)=\begin{cases}
r & j=e\\
0 & j\ne e
\end{cases}.
\]
\end{conj}

Note that any finitely presented $\widehat{S}$-module has a presentation matrix that involves only finitely many variables.  Hence the Betti table of every $\widehat{S}$-module can be decomposed as a sum of the Betti tables of Cohen--Macaulay modules with pure resolutions.  Thus, by taking the union of the cones of Betti tables of resolutions over $\kk[x_0,x_1,\dots,x_n]$ as $n\to \infty$, we will obtain the cone $\BQ(\widehat{S})$ of resolutions of Betti tables of finitely presented $\widehat{S}$-modules.

One could then imagine trying to prove Stillman's Conjecture via Boij--S\"oderberg theory by showing that there are only finitely many {\em lattice points} in $\BQ(\widehat{S})$  that satisfy the conditions of the conjecture.  However, this turns out to be false.

\begin{prop}\label{prop:virtual}
Fix any $e\geq 1$ and $r\geq 2$. Then there exist infinitely many lattice points $D\in \BQ(\widehat{S})$ which satisfy
\[
\beta_{0,j}(D)=\begin{cases}
1 & j=0\\
0 & j\ne 0
\end{cases}
\qquad
\text{ and }
\qquad
\beta_{1,j}(D)=\begin{cases}
r & j=e\\
0 & j\ne e
\end{cases}.
\]
In other words, there exist infinitely many lattice points in $D\in\BQ(\widehat{S})$ that look like the Betti table of an algebra generated by $r$ forms of degree $e$. In fact, there exist infinitely many $D$ satisfying the above and which also look like pure resolutions.
\end{prop}

Yet this may not be the end of the story.  Remark~\ref{rmk:obstruction} shows that the ``virtual'' Betti tables constructed in the proof of the proposition cannot arise as the Betti table of an actual module, and we know of no other such infinite families.

\begin{question}\label{question:still 2}
Ignoring pure Betti table counterexamples, like those constructed in Proposition~\ref{prop:virtual}, do there exist finitely many lattice points in $D\in\BQ(\widehat{S})$ that look like the Betti table of an algebra generated by $r$ forms of degree $e$?
\end{question}

A positive answer to this question would certainly be interesting as it would imply Stillman's Conjecture.  But a negative answer would also be interesting for a different reason: assuming that we also believe Stillman's Conjecture, a negative answer to Question~\ref{question:still 2} would suggest that almost all lattice points in the cone of Betti tables that ``look like'' the Betti table of a cyclic module do not come from an actual Betti table.  In other words, a negative answer would suggest that the ``noise'' of the fake Betti tables inside the cone overwhelms the ``signal'' of the actual Betti tables, at least for some questions.

\begin{proof}[Proof of Proposition~\ref{prop:virtual} $($sketch$)$]
Fix $p\geq 0$ and let $n:= r+p(r-1)$.  For each such $p$, we will define a degree sequence $d^{(p)} \in \ZZ^{n+1}$, and we will show that for infinitely many of these degree sequences, the smallest integral pure diagram of type $d^{(p)}$ satisfies the conditions of the proposition.  We set $d^{(p)}_0=0$ and $d^{(p)}_1=e$.  For $i=2, \dots, n$ we then set $d^{(p)}_i:=e(p+i)$.  

We let $\pi$ be the normalized pure diagram of type $d^{(p)}$, so that $\beta_{0,0}(\pi)=1$ by assumption.  It follows from~\cite[\S2.1]{boij-sod1} that
\[
\beta_{i,d^{(p)}_i} = \frac{\prod_{j\ne 0}^n d^{(p)}_j}{\prod_{j\ne i}^n |d^{(p)}_i-d^{(p)}_j| } .
\]
To complete the proof we must show that each of these expressions is an integer.  Since every entry of $d^{(p)}$ is divisible by $e$, we can reduce to the case where $e=1$.  A detailed but elementary computation then confirms that each of these expressions is an integer.
\end{proof}

\begin{remark}\label{rmk:obstruction}
For a given $r\geq 2$ and $p>0$, the virtual Betti table constructed in the proof of Proposition~\ref{prop:virtual} would correspond to a module with codimension $r+p(r-1)$.  But this would be the Betti table of a module $\beta(S/I)$ where $I$ is an ideal with $r$ generators.  In particular, the codimension of $S/I$ is bounded above by $r$, which is a contradiction.  Thus, the pure Betti tables constructed in the proof of Proposition~\ref{prop:virtual} cannot correspond to an actual Betti table.  Of course, some scalar multiple of each table does come from an actual Betti table.
\end{remark}

\begin{example}
We consider the examples from the proof of Proposition~\ref{prop:virtual} in the case $e=2$ and $r=3$.  These would correspond to algebras $S/I$ where $I$ is generated by three quadrics. For any prime $p$, there is a pure diagram with degree sequence
\[
d^{(p)}:= (0,2,4+2p,6+2p,\dots,6+6p)\in \ZZ^{2p+4}
\]
that looks like the Betti table of such an algebra.   With $p=3$, this yields
\[
\begin{small}
\begin{pmatrix}
1&&&&&&&\\
&3&&&&&&\\
&&&&&&&&\\
&&&&&&&&\\
&&42&&&&&\\
&&&126&&&&\\
&&&&168&&&\\
&&&&&120&&\\
&&&&&&45&\\
&&&&&&&7\\
\end{pmatrix}.
\end{small}
\]
Yet a result of Eisenbud and Huneke implies that such an algebra has projective dimension at most $4$~\cite[Theorem~3.1]{mccullough-seceleanu}, and so Stillman's Conjecture is true in this case, despite the presence of these ``fake'' Betti tables. 
\end{example}

\section{Extremal rays}\label{sec:extremal rays}
A number of questions remain about the extremal rays.  To begin with, Eisenbud, Fl{\o}ystad, and Weyman conjecture that every sufficiently large integral point on an extremal ray comes from an actual Betti table~\cite[Conjecture~6.1]{efw}.  This conjecture remains open, though it is known to be false for interior rays in the cone~\cite[Example~1.7]{ees-filtering}.  One could ask a similar question for cohomology tables.

In addition, $\GL$-equivariance plays a serendipitous role in Boij--S\"oderberg theory.  We restrict to characteristic zero and to the finite length cone $\BQ^{n+1}(S)$ and the vector bundle cone $\CvbQ(\PP^n)$.  One can construct each extremal ray of these two cones by a $\GL_{n+1}$-equivariant object: see~\cite[\S3]{efw} for the equivariant pure resolutions and~\cite[Theorem~6.2]{eis-schrey1} for the equivariant supernatural bundles.  An immediate corollary is that {\em every} ray in $\BQ^{n+1}(S)$ or $\CvbQ(\PP^n)$ can be realized by an equivariant object:

\begin{cor}\label{cor:equivariant multiple} Let $\kk$ be a field of characteristic $0$ and let $S=\kk[x_0,\dots,x_n]$.  
\begin{enumerate}[\indent \rm (1)]
\item For any finite length, graded, $S$-module $M$, there exists a finite length, graded, and $\GL_{n+1}$-equivariant module $N$ such that $\beta(N)$ is a scalar multiple of $\beta(M)$.
\item  For any vector bundle $\cE$ on $\PP^n_k$, there is a $\GL_{n+1}$-equivariant vector bundle $\cF$ on $\PP^n$ such that $\gamma(\cE)$ is a scalar multiple of $\gamma(\cF)$.
\end{enumerate}
\end{cor}

\begin{question}
Is there an intrinsic reason why every Betti table of a finite length module, and every cohomology table of a vector bundle, can be realized $($up to scalar multiple$)$ by an equivariant object?  Is there a simpler proof of this fact?
\end{question}

\section{More topics}\label{sec:more topics}

\subsection{Toric Boij--S\"oderberg theory}  
Many researchers have observed that it would be natural to try to extend \BS theory to toric varieties and free complexes over their multigraded Cox rings.  Eisenbud and Schreyer have made a conjecture about the extremal rays of the cone of cohomology on  $\PP^1\times \PP^1$~\cite{eis-schrey-abel}; several authors have partial results for cones of Betti tables over a polynomial ring with $\ZZ^2$ or $\ZZ^s$ grading~\cite{boij-floystad,floystad-multigraded,bek-flavors}; and ~\cite[\S11]{eis-erman-cat} develops some of the duality aspects for toric varieties.  Yet even taken together, these results are far from providing a complete picture for any toric variety, and it appears to be quite challenging to develop such a picture even for $\PP^1\times \PP^1$.  A different possible direction is to extend Boij and Smith's work on cones of Hilbert functions to the multigraded case; see their remarks in~\cite[p.~10,317]{boij-smith}.

\subsection{Equivariant Boij--S\"oderberg theory}

Since pure free resolutions over a polynomial ring have equivariant realizations over a field of characteristic $0$, one can ask whether there is a meaningful description of the cone of ``equivariant Betti tables''.  See \cite[\S 4]{sam-weyman} for the setup and partial progress.  A natural variant is to work on a Grassmannian and to focus on $\GL$-equivariant free resolutions and $\GL$-equivariant cohomology tables.  There is work in progress in this direction due to Ford--Levinson and Ford--Levinson--Sam.

\subsection{Monomial ideals and combinatorics}
A number of authors have explored the boundary between Boij--S\"oderberg theory and combinatorial commutative algebra.  One remarkable recent result is Mayes-Tang's proof \cite{mayes-tang} of Engstr{\"o}m's Conjecture on stabilization of decompositions of powers of monomial ideals~\cite{engstrom}.  This raises the question of whether similar asymptotic stabilization results can be expected in other contexts.

In another direction, Fl\o ystad \cite{floystad-triplets} has proposed a vast generalization of Boij--S\"oderberg cones by moving from Betti tables to certain triples of homological data which are defined on the category of squarefree monomial modules.  This in turn has relations to the cone of hypercohomology tables of complexes of coherent sheaves and Tate resolutions~\cite{floystad-zipping}, and Fl\o ystad formulates many fascinating questions in those papers.

See~\cite{cook, engstrom-stamps, gibbons, nagel-sturgeon} for some other work that mixes Boij-S\"oderberg theory and combinatorics.

\subsection{Asymptotic Boij--S\"oderberg decompositions of Veroneses}  Let $X$ be a smooth projective variety with a very ample divisor $A$.  For any $d>0$ we can embed $X\subseteq \PP^{r_d}$ by the complete liner series $|dA|$ and study the syzygies of $X$.  Ein and Lazarsfeld ask: as $d\to \infty$, which Betti numbers of $X$ will be nonzero? They show that in the limit, the answer only depends on the dimension of the variety of $X$, and not on any more refined geometric properties about $X$~\cite{ein-lazarsfeld-asymptotic}.  A natural followup is to ask about the asymptotic behavior of the Boij--S\"oderberg decomposition of these Betti tables.  See~\cite[Problem~3.6]{ein-erman-lazarsfeld} and~\cite{erman-high-degree-curve}.

\section*{Acknowledgments}
We would like to thank the participants of the \BS Theory group at the Bootcamp for the 2015 Algebraic Geometry Summer Research Institute, whose ideas and questions informed this note.  Several of these questions were developed in speculative conversations with colleagues, and we also thank those individuals, including: Christine Berkesch Zamaere, Mats Boij, David Eisenbud, Courtney Gibbons, Jason McCullough, Frank-Olaf Schreyer, Greg Smith, Jerzy Weyman, and many more.


\begin{bibdiv}
\begin{biblist}
\begin{small}

\bib{artin-schelter}{article}{
   author={Artin, Michael},
   author={Schelter, William F.},
   title={Graded algebras of global dimension $3$},
   journal={Adv. in Math.},
   volume={66},
   date={1987},
   number={2},
   pages={171--216},
}

\bib{avramov1}{article}{
   author={Avramov, Luchezar L.},
   title={Infinite free resolutions},
   conference={
      title={Six lectures on commutative algebra},
      address={Bellaterra},
      date={1996},
   },
   book={
      series={Progr. Math.},
      volume={166},
      publisher={Birkh\"auser},
      place={Basel},
   },
   date={1998},
   pages={1--118},
}

\bib{avramov-buchweitz-betti}{article}{
   author={Avramov, Luchezar L.},
   author={Buchweitz, Ragnar-Olaf},
   title={Lower bounds for Betti numbers},
   journal={Compositio Math.},
   volume={86},
   date={1993},
   number={2},
   pages={147--158},
}

\bib{bbeg}{article}{
      author={Berkesch, Christine},
      author={Burke, Jesse},
      author={Erman, Daniel},
      author={Gibbons, Courtney},
      title={The cone of Betti diagrams over a hypersurface ring of low embedding dimension},
      journal={J. Pure Appl. Algebra},
note={\arXiv{1109.5198v2}},
   volume={16},
   date={2012},
   pages={121--141},
   }

\bib{bek-flavors}{article}{
   author={Berkesch, Christine},
   author={Erman, Daniel},
   author={Kummini, Manoj},
   title={Three flavors of extremal Betti tables},
   conference={
      title={Commutative algebra},
   },
   book={
      publisher={Springer, New York},
   },
note={\arXiv{1207.5707v1}},
   date={2013},
   pages={99--121},
}

\bib{beks-local}{article}{
   author={Berkesch, Christine},
   author={Erman, Daniel},
   author={Kummini, Manoj},
   author={Sam, Steven V},
   title={Shapes of free resolutions over a local ring},
   journal={Math. Ann.},
   volume={354},
   date={2012},
   number={3},
   pages={939--954},
      note={\arXiv{1105.2244v2}},
}



\bib{beks-tensor}{article}{
   author={Berkesch Zamaere, Christine},
   author={Erman, Daniel},
   author={Kummini, Manoj},
   author={Sam, Steven V},
   title={Tensor complexes: multilinear free resolutions constructed from
   higher tensors},
   journal={J. Eur. Math. Soc. (JEMS)},
   volume={15},
   date={2013},
   number={6},
   pages={2257--2295},
   note={\arXiv{1101.4604v5}},
}



\bib{boij-floystad}{article}{
   author={Boij, Mats},
   author={Fl{\o}ystad, Gunnar},
   title={The cone of Betti diagrams of bigraded Artinian modules of
   codimension two},
   conference={
      title={Combinatorial aspects of commutative algebra and algebraic
      geometry},
   },
   book={
      series={Abel Symp.},
      volume={6},
      publisher={Springer},
      place={Berlin},
   },
   date={2011},
   pages={1--16},
note={\arXiv{1001.3238v1}},
}

\bib{boij-smith}{article}{
   author={Boij, Mats},
   author={Smith, Gregory G.},
   title={Cones of Hilbert functions},
   journal={Int. Math. Res. Not. IMRN},
   date={2015},
   number={20},
   pages={10314--10338},
   note = {\arXiv{1404.7341v1}}
}


\bib{boij-sod1}{article}{
    AUTHOR = {Boij, Mats},
    AUTHOR = {S{\"o}derberg, Jonas},
     TITLE = {Graded {B}etti numbers of {C}ohen-{M}acaulay modules and the
              multiplicity conjecture},
   JOURNAL = {J. Lond. Math. Soc. (2)},
  FJOURNAL = {Journal of the London Mathematical Society. Second Series},
    VOLUME = {78},
      YEAR = {2008},
    NUMBER = {1},
     PAGES = {85--106},
note={\arXiv{math/0611081v2}},
}

\bib{boij-sod2}{article}{
   author={Boij, Mats},
   author={S{\"o}derberg, Jonas},
   title={Betti numbers of graded modules and the multiplicity conjecture in
   the non-Cohen-Macaulay case},
   journal={Algebra Number Theory},
   volume={6},
   date={2012},
   number={3},
   pages={437--454},
note = {\arXiv{0803.1645v1}},
}


\bib{boocher}{article}{
   author={Boocher, Adam},
   title={Free resolutions and sparse determinantal ideals},
   journal={Math. Res. Lett.},
   volume={19},
   date={2012},
   number={4},
   pages={805--821},
note={\arXiv{1111.0279v2}},
}

\bib{buchs-eis-gor}{article}{
    AUTHOR = {Buchsbaum, David A.},
    AUTHOR = {Eisenbud, David},
     TITLE = {Algebra structures for finite free resolutions, and some
              structure theorems for ideals of codimension {$3$}},
   JOURNAL = {Amer. J. Math.},
  FJOURNAL = {American Journal of Mathematics},
    VOLUME = {99},
      YEAR = {1977},
    NUMBER = {3},
     PAGES = {447--485},
}

\bib{chara-evans-problems}{incollection}{
    AUTHOR = {Charalambous, H.},
    AUTHOR = {Evans Jr., E. G.},
     TITLE = {Problems on {B}etti numbers of finite length modules},
 BOOKTITLE = {Free resolutions in commutative algebra and algebraic geometry
              ({S}undance, {UT}, 1990)},
    SERIES = {Res. Notes Math.},
    VOLUME = {2},
     PAGES = {25--33},
 PUBLISHER = {Jones and Bartlett},
   ADDRESS = {Boston, MA},
      YEAR = {1992},
}

 \bib{cook}{article}{
    author={Cook, David, II},
    title={The structure of the Boij--S\"oderberg posets},
    journal={Proc. Amer. Math. Soc.},
    volume={139},
    date={2011},
    number={6},
    pages={2009--2015},
note={\arXiv{1006.2026v1}},
 }

\bib{ein-erman-lazarsfeld}{article}{
   author={Ein, Lawrence},
   author={Erman, Daniel},
   author={Lazarsfeld, Robert},
   title={Asymptotics of random Betti tables},
   journal={J. Reine Angew. Math.},
   volume={702},
   date={2015},
   pages={55--75},
note={\arXiv{1207.5467v1}},
}

\bib{ein-lazarsfeld-asymptotic}{article}{
   author={Ein, Lawrence},
   author={Lazarsfeld, Robert},
   title={Asymptotic syzygies of algebraic varieties},
   journal={Invent. Math.},
   volume={190},
   date={2012},
   number={3},
   pages={603--646},
	note = {\arXiv{1103.0483v3}},
}

\bib{eis-syzygy}{book}{
author={Eisenbud, David},
   title={The geometry of syzygies},
   series={Graduate Texts in Mathematics},
   volume={229},
   note={A second course in commutative algebra and algebraic geometry},
   publisher={Springer-Verlag},
   place={New York},
   date={2005},
   pages={xvi+243},
}



\bib{eis-erman-cat}{article}{
AUTHOR = {Eisenbud, David},
AUTHOR = {Erman, Daniel},
TITLE = {Categorified duality in Boij--S\"{o}derberg theory and invariants of free complexes},
note = {\arXiv{1205.0449v2}},
year = {2012},
}

\bib{ees-filtering}{article}{
   author={Eisenbud, David},
   author={Erman, Daniel},
   author={Schreyer, Frank-Olaf},
   title={Filtering free resolutions},
   journal={Compos. Math.},
   volume={149},
   date={2013},
   number={5},
   pages={754--772},
note={\arXiv{1001.0585v3}},
}
	


\bib{eis-floy-schrey}{article}{
   author={Eisenbud, David},
   author={Fl{\o}ystad, Gunnar},
   author={Schreyer, Frank-Olaf},
   title={Sheaf cohomology and free resolutions over exterior algebras},
   journal={Trans. Amer. Math. Soc.},
   volume={355},
   date={2003},
   number={11},
   pages={4397--4426},
note={\arXiv{math/0104203v2}},
}

\bib{efw}{article}{
AUTHOR = {Eisenbud, David},
AUTHOR = {Fl{\o}ystad, Gunnar},
AUTHOR = {Weyman, Jerzy},
TITLE = {The existence of equivariant pure free resolutions},
JOURNAL = {Annales de l'institut Fourier},
volume = {61},
NUMBER = {3},
date = {2011},
pages={905--926},
note={\arXiv{0709.1529v5}},
}

\bib{eisenbud-peeva}{article}{
	   author={Eisenbud, David},
	      author={Peeva, Irena},
	title = {Matrix Factorizations for Complete Intersections and Minimal Free Resolutions},
	year = {2013},
	note = {\arXiv{1306.2615v4}},
	}
	

\bib{eis-schrey-chow}{article}{
   author={Eisenbud, David},
   author={Schreyer, Frank-Olaf},
   title={Resultants and Chow forms via exterior syzygies},
   journal={J. Amer. Math. Soc.},
   volume={16},
   date={2003},
   number={3},
   pages={537--579},
note={\arXiv{math/0111040v1}},
}

\bib{eis-schrey1}{article}{
   author={Eisenbud, David},
   author={Schreyer, Frank-Olaf},
   title={Betti numbers of graded modules and cohomology of vector bundles},
   journal={J. Amer. Math. Soc.},
   volume={22},
   date={2009},
   number={3},
   pages={859--888},
   issn={0894-0347},
note={\arXiv{0712.1843v4}},
}

\bib{eis-schrey2}{article}{
   author={Eisenbud, David},
   author={Schreyer, Frank-Olaf},
   title={Cohomology of coherent sheaves and series of supernatural bundles},
   journal={J. Eur. Math. Soc. (JEMS)},
   volume={12},
   date={2010},
   number={3},
   pages={703--722},
note={\arXiv{0902.1594v1}},
}



\bib{eis-schrey-icm}{inproceedings}{
      author={Eisenbud, David},
      author={Schreyer, Frank-Olaf},
       title={Betti numbers of syzygies and cohomology of coherent sheaves},
        date={2010},
   booktitle={Proceedings of the {I}nternational {C}ongress of
  {M}athematicians},
        note={Hyderabad, India, \arXiv{1102.3559v1}},
}


\bib{eis-schrey-abel}{article}{
   author={Eisenbud, David},
   author={Schreyer, Frank-Olaf},
   title={Boij--S\"oderberg theory},
   conference={
      title={Combinatorial aspects of commutative algebra and algebraic
      geometry},
   },
   book={
      series={Abel Symp.},
      volume={6},
      publisher={Springer},
      place={Berlin},
   },
   date={2011},
   pages={35--48},
}

\bib{eis-schrey-banks}{article}{
   author={Eisenbud, David},
   author={Schreyer, Frank-Olaf},
   title={The banks of the cohomology river},
   journal={Kyoto J. Math.},
   volume={53},
   date={2013},
   number={1},
   pages={131--144},
note={\arXiv{1109.4591v1}},
}


\bib{engstrom}{article}{
   author={Engstr{\"o}m, Alexander},
   title={Decompositions of Betti diagrams of powers of monomial ideals: a
   stability conjecture},
   conference={
      title={Combinatorial methods in topology and algebra},
   },
   book={
      series={Springer INdAM Ser.},
      volume={12},
      publisher={Springer, Cham},
   },
   date={2015},
note={\arXiv{1312.6981v1}},
pages={37--40},
}


\bib{engstrom-stamps}{article}{
   author={Engstr{\"o}m, Alexander},
   author={Stamps, Matthew T.},
   title={Betti diagrams from graphs},
   journal={Algebra Number Theory},
   volume={7},
   date={2013},
   number={7},
   pages={1725--1742},
note={\arXiv{1210.8069v3}},
}

\bib{erman-beh}{article}{
   author={Erman, Daniel},
   title={A special case of the Buchsbaum-Eisenbud-Horrocks rank conjecture},
   journal={Math. Res. Lett.},
   volume={17},
   date={2010},
   number={6},
   pages={1079--1089},
note={\arXiv{0902.0316v2}},
}


\bib{erman-high-degree-curve}{article}{
   author={Erman, Daniel},
   title={The Betti table of a high-degree curve is asymptotically pure},
   conference={
      title={Recent advances in algebraic geometry},
   },
   book={
      series={London Math. Soc. Lecture Note Ser.},
      volume={417},
      publisher={Cambridge Univ. Press, Cambridge},
   },
   date={2015},
   pages={200--206},
note={\arXiv{1308.4661v2}},
}

\bib{erman-sam-diagonals}{article}{
   author={Erman, Daniel},
author={Sam, Steven V},
   title={Supernatural analogues of Beilinson monads},
journal={Compos. Math.},
note={\arXiv{1506.07558v1}},
}


\bib{floystad-multigraded}{article}{
   author={Fl{\o}ystad, Gunnar},
   title={The linear space of Betti diagrams of multigraded Artinian
   modules},
   journal={Math. Res. Lett.},
   volume={17},
   date={2010},
   number={5},
   pages={943--958},
note={\arXiv{1001.3235v2}},
}
	

\bib{floystad-expository}{article}{
	AUTHOR = {Fl{\o}ystad, Gunnar},
     TITLE = {Boij--S\"oderberg theory: Introduction and survey},
     NOTE = {\arXiv{1106.0381v2}},
booktitle={Progress in Commutative Algebra 1, Combinatorics and homology},
series={Proceedings in mathematics},
publisher={du Gruyter}, 
pages={1-54},
YEAR = {2012},
     }
     
     \bib{floystad-zipping}{article}{
   author={Fl{\o}ystad, Gunnar},
   title={Zipping Tate resolutions and exterior coalgebras},
   journal={J. Algebra},
   volume={437},
   date={2015},
   pages={249--307},
note={\arXiv{1212.3675v3}},
}
		
\bib{floystad-triplets}{article}{
   author={Fl{\o}ystad, Gunnar},
   title={Triplets of pure free squarefree complexes},
   journal={J. Commut. Algebra},
   volume={5},
   date={2013},
   number={1},
   pages={101--139},
note={\arXiv{1207.2071v3}},
}

\bib{survey2}{article}{
title={Three themes of syzygies},
author={Fl\o ystad, Gunnar},
author={McCullough, Jason},
author={Peeva, Irena},
journal={Bull. Amer. Math. Soc.},
volume={53},
date={2016}, 
pages={415--435},
}


\bib{gheorghita-sam}{article}{
	author = {Iulia Gheorghita},
	author = {Steven V Sam},
	title = {The cone of Betti tables over three non-collinear points in the plane},	
journal={J. Commut. Alg.},
note={\arXiv{1501.00207v1}},
}

\bib{gibbons}{article}{
   author={Gibbons, Courtney},
   author={Jeffries, Jack},
   author={Mayes, Sarah},
   author={Raicu, Claudiu},
   author={Stone, Branden},
   author={White, Bryan},
   title={Non-simplicial decompositions of Betti diagrams of complete
   intersections},
   journal={J. Commut. Algebra},
   volume={7},
   date={2015},
   number={2},
   pages={189--206},
note={\arXiv{1301.3441v3}},
}

\bib{M2}{misc}{
    label={M2},
    author={Grayson, Daniel~R.},
    author={Stillman, Michael~E.},
    title = {Macaulay 2, a software system for research
	    in algebraic geometry},
    note = {Available at \url{http://www.math.uiuc.edu/Macaulay2/}},
}

\bib{hartshorne-vector}{article}{
    AUTHOR = {Hartshorne, Robin},
     TITLE = {Algebraic vector bundles on projective spaces: a problem list},
   JOURNAL = {Topology},
  FJOURNAL = {Topology. An International Journal of Mathematics},
    VOLUME = {18},
      YEAR = {1979},
    NUMBER = {2},
     PAGES = {117--128},
}


\bib{herzog-srinivasan}{article}{
   author={Herzog, J{\"u}rgen},
   author={Srinivasan, Hema},
   title={Bounds for multiplicities},
   journal={Trans. Amer. Math. Soc.},
   volume={350},
   date={1998},
   number={7},
   pages={2879--2902},
}

\bib{kummini-sam}{article}{
	author = {Manoj Kummini},
	author = {Steven V Sam},
	title = {The cone of Betti tables over a rational normal curve},
booktitle={Commutative Algebra and Noncommutative Algebraic Geometry}, 
pages={251--264}, 
series={Math. Sci. Res. Inst. Publ.},
volume={68}, 
publisher={Cambridge Univ. Press, Cambridge},
date={2015},
note={\arXiv{1301.7005v2}},
}

\bib{mayes-tang}{article} {
author={Mayes-Tang, Sarah},
title={Stabilization of Boij-S\"oderberg decompositions of ideal powers},
note={\arXiv{1509.08544v1}},
}


\bib{mccullough-concavity}{article}{
   author={McCullough, Jason},
   title={A polynomial bound on the regularity of an ideal in terms of half
   of the syzygies},
   journal={Math. Res. Lett.},
   volume={19},
   date={2012},
   number={3},
   pages={555--565},
note={\arXiv{1112.0058v1}},
}

		
\bib{mccullough-seceleanu}{article}{
   author={McCullough, Jason},
   author={Seceleanu, Alexandra},
   title={Bounding projective dimension},
   conference={
      title={Commutative algebra},
   },
   book={
      publisher={Springer, New York},
   },
   date={2013},
   pages={551--576},
}

\bib{nagel-sturgeon}{article}{
   author={Nagel, Uwe},
   author={Sturgeon, Stephen},
   title={Combinatorial interpretations of some Boij-S\"oderberg
   decompositions},
   journal={J. Algebra},
   volume={381},
   date={2013},
   pages={54--72},
note={\arXiv{1203.6515v1}},
}


\bib{peeva-book}{book}{
   author={Peeva, Irena},
   title={Graded syzygies},
   series={Algebra and Applications},
   volume={14},
   publisher={Springer-Verlag London Ltd.},
   place={London},
   date={2011},
   pages={xii+302},
   isbn={978-0-85729-176-9},
}



\bib{peeva-stillman}{article}{
   author={Peeva, Irena},
   author={Stillman, Mike},
   title={Open problems on syzygies and Hilbert functions},
   journal={J. Commut. Algebra},
   volume={1},
   date={2009},
   number={1},
   pages={159--195},
}


\bib{sam}{article} {
   author={Sam, Steven V},
   title={Homology of analogues of Heisenberg Lie algebras},
   journal={Math. Res. Lett.},
   volume={22},
   date={2015},
   number={4},
   pages={1223--1241},
note={\arXiv{1307.1901v2}},
}

\bib{sam-weyman}{article}{
   author={Sam, Steven V},
   author={Weyman, Jerzy},
   title={Pieri resolutions for classical groups},
   journal={J. Algebra},
   volume={329},
   date={2011},
   pages={222--259},
note={\arXiv{0907.4505v5}},
}
\end{small}


\end{biblist}
\end{bibdiv}
\end{document}